\documentclass[11pt,a4paper]{article}

\usepackage[utf8]{inputenc}
\usepackage{amsmath}
\usepackage{amsfonts}
\usepackage{amssymb}
\usepackage{amsthm}
\usepackage{bbm}
\usepackage{xcolor}
\usepackage{dsfont}
\usepackage{graphicx}
\usepackage{subfig}
\usepackage{hyperref}
\usepackage[top=2cm,
  bottom=2cm,
  left=2cm,
  right=2cm]{geometry}
\numberwithin{equation}{section}

\newtheorem{theorem}{Theorem}[section]

\newtheorem{definition}[theorem]{Definition}

\newtheorem{proposition}[theorem]{Proposition}

\theoremstyle{definition}
\newtheorem{example}[theorem]{Example}
\newtheorem{remark}[theorem]{Remark}

\newcommand{\E}{\mathbb{E}}
\newcommand{\calA}{\mathcal A}
\newcommand{\calH}{\mathcal H}
\newcommand{\R}{\mathbb R}

\newcommand{\calZ}{\mathcal Z}

\newcommand{\bbH}{\mathbb H}
\renewcommand{\t}{t\in[0,T]}

\title{Optimal control in linear-quadratic stochastic advertising models with memory}

\author{Michele Giordano$^{1}$\\\href{mailto:michelgi@math.uio.no}{michelgi@math.uio.no}
   \and Anton Yurchenko-Tytarenko$^1$ \\ \href{mailto:antony@math.uio.no}{antony@math.uio.no}}

\date{
    $^1$Department of Mathematics, University of Oslo
}

\begin{document}

\maketitle

\begin{abstract}
    This paper deals with a class of optimal control problems which arises in advertising models with Volterra Ornstein-Uhlenbeck process representing the product goodwill. Such choice of the model can be regarded as a stochastic modification of the classical Nerlove-Arrow model that allows to incorporate both presence of uncertainty and empirically observed memory effects such as carryover or distributed forgetting. We present an approach to solve such optimal control problems based on an infinite dimensional lift which allows us to recover Markov properties by formulating an optimization problem equivalent to the original one in a Hilbert space. Such technique, however, requires the Volterra kernel from the forward equation to have a representation of a particular form that may be challenging to obtain in practice. We overcome this issue for H\"older continuous kernels by approximating them with Bernstein polynomials, which turn out to enjoy a simple representation of the required type. Then we solve the optimal control problem for the forward process with approximated kernel instead of the original one and study convergence. The approach is illustrated with simulations.
\end{abstract}

\textbf{Keywords:} \textit{Dynamic programming, Volterra Ornstein-Uhlenbeck process,
infinite-dimensional Bellman equations,  optimal advertising}
    
    \textbf{MSC 2020:} {60H20, 92E20, 91B70, 90B60}

\paragraph{Funding.} The present research is carried out within the frame and support of the ToppForsk project nr. 274410 of the Research Council of Norway with title STORM: Stochastics for Time-Space Risk Models.

\section{Introduction}

The problem of optimizing advertising strategies has always been of paramount importance in the field of marketing. Starting from the pioneering works of Vidale and Wolfe \cite{VW1957} and Nerlove and Arrow \cite{NA1962}, this topic has evolved into a full-fledged field of research and modeling. Realizing the impossibility of describing all existing classical approaches and results, we refer the reader to the review article of Sethi \cite{S1977} (that analyzes the literature prior to 1975) and a more recent paper by Feichtinger, Hartl and Sethi \cite{FHS1994} (covering the results up to 1994) and references therein.

It is worth noting that the Nerlove--Arrow approach, which was the foundation for numerous modern dynamic advertising models, assumed no time lag between spending on advertising and the impact of the latter on the goodwill stock. However, many empirical studies (see, for example, \cite{Leone1995}) clearly indicate some kind of a ``memory'' phenomenon that is often called the ``distributed lag'' or ``carryover'' effect: the influence of advertising does not have an immediate impact but is rather spread over a period of time varying from several weeks to several months. This shortcoming of the basic Nerlove--Arrow model gave rise to many modifications of the latter aimed at modeling distributed lags. For a long time, nevertheless, the vast majority of dynamic advertising models with distributed lags had been formulated in a deterministic framework (see e.g. \cite[\S 2.6]{S1977} and \cite[Section 2.3]{FHS1994}).

In recent years, however, there have been several landmark papers that consider the Nerlove-Arrow-type model with memory in a stochastic setting. Here, we refer primarily to the series of papers \cite{GM2006,GMS2009} (see also a more recent work \cite{LZh2020}), where goodwill stock is modeled via Brownian linear diffusion with delay of the form
\begin{equation}\label{eq: intro SDDE}
    dX^u(t) = \left( \alpha_0 X^u(t) + \int_{-r}^0 \alpha_1(s) X^u(t+s)ds + \beta_0 u(t) + \int_{-r}^0 \beta_1(s) u(t+s) ds \right)dt + \sigma dW(t),
\end{equation}
where $X^u$ is interpreted as the product's goodwill stock and $u$ is the spending on advertising. The corresponding optimal control problem in this case was solved using the so-called lift approach: equation \eqref{eq: intro SDDE} was rewritten as a stochastic differential equation (without delay) in a suitable Hilbert space, and then infinite-dimensional optimization techniques (either dynamic programming principle or maximum principle) were applied.

In this article, we present an alternative stochastic model that also takes the carryover effect into account. Instead of the delay approach described above, we incorporate the memory into the model by means of the Volterra kernel $K \in L^2([0,T])$ and consider the controlled Volterra Ornstein-Uhlenbeck process of the form 
\begin{equation}\label{Forward_introduction}
    X^u(t)=X(0)+\int_0^t K(t-s)\Big(\alpha u(s)-\beta X^u(s)\Big)ds+\sigma\int_0^t K(t-s) dW(s),
\end{equation}
where $\alpha,\beta,\sigma > 0$ and $X(0) \in \mathbb R$ are constants (see e.g. \cite[Section 5]{AffineVolterraProcesses} for more details on affine Volterra processes of such type). Note that such goodwill dynamics can be regarded as the combination of deterministic lag models described in \cite[Section 2.3]{FHS1994} and the stochastic Ornstein-Uhlenbeck-based model presented by Rao \cite{Rao1986}. The main difference from \eqref{eq: intro SDDE} is the memory incorporated to the noise along with the drift as the stochastic environment (represented by the noise) tends to form ``clusters'' with time. Indeed, in reality positive increments are likely to be followed by positive increments (if conditions are favourable for the goodwill during some period of time) and negative increments tend to follow negative increments (under negative conditions). This behaviour of the noise cannot be reflected by a standard Brownian driver but can easily be incorporated into the model \eqref{Forward_introduction}.

Our goal is to solve an optimization problem of the form
\begin{equation}\label{eq: optimization problem}
    \begin{cases}
        X^u(t)=X(0)+\int_0^t K(t-s)\Big(\alpha u(s)-\beta X^u(s)\Big)ds+\sigma\int_0^t K(t-s) dW(s),
        \\
        J(u) := \E\left[-\int_0^Ta_1u^2(s)ds+a_2X^u(T)\right] \to \max,
    \end{cases}
\end{equation}
where $a_1,a_2 > 0$ are given constants. The set of admissible controls for the problem \eqref{eq: optimization problem}, denoted by $L^2_a := L^2_a(\Omega\times[0,T])$, is the space of square integrable real-valued stochastic processes adapted to the filtration generated by $W$. Note that the process $X^u$ is well defined for any $u\in L^2_a$ since, for almost all $\omega\in\Omega$, the equation \eqref{Forward_introduction} treated pathwisely can be considered as a deterministic linear Volterra integral equation of the second kind that has a unique solution (see e.g. \cite{Tricomi_1985}).  

The optimization problem \eqref{eq: optimization problem} for underlying Volterra dynamics has been studied by several authors (see, e.g. \cite{AO,Yong} and the bibliography therein). Contrarily to most of the work in our bibliography, we will not solve such problem by means of a maximum principle approach. Even though this method allows to find necessary and sufficient conditions to obtain the optimal control to \eqref{eq: optimization problem}, we cannot directly apply it as we deal with low regularity conditions on the coefficients of our drift and volatility. Furthermore, such method has another notable drawback in the practice. In fact, its application is often associated with computations of conditional expectations that are substantially challenging due to the absence of Markovianity. Another possible method to solve the optimal control problem \eqref{eq: optimization problem} is to get an explicit solution of the forward equation \eqref{Forward_introduction}, plug it into the performance functional and try to solve the maximization problem using differential calculus in Hilbert spaces. But, even though this method seems appealing, obtaining the required explicit representation of $X^u$ in terms of $u$ might be tedious and burdensome. Instead, we will use the approach introduced in \cite{Pham,DNG_Lift} that is in the same spirit of the one in \cite{GM2006,GMS2009,LZh2020} mentioned above: we will rewrite the original forward stochastic Volterra integral equation as a stochastic differential equation in a suitable Hilbert space and then apply standard optimization techniques in infinite dimensions (see e.g. \cite{FGSw2017,FT}). Moreover, the shape of the corresponding infinite-dimensional Hamilton-Jacobi-Bellman equation allows to obtain an explicit solution to the latter by exploiting the ``splitting'' method from \cite[Section 3.3]{GMS2009}.

We notice that, while the optimization problem \eqref{eq: optimization problem} is closely related to the one presented in \cite{Pham}, there are several important differences in comparison to our work. In particular, \cite{Pham} demands the kernel to have the form \begin{equation}\label{Kernel Laplace transform}
    K(t) = \int_{\mathbb R_+} e^{-\theta t} \mu(d\theta),
\end{equation}
where $\mu$ is a signed measure such that $\int_{\mathbb R_+} (1 \wedge \theta^{-1/2}) |\mu|(d\theta) < \infty$. Although there are some prominent examples of such kernels, not all kernels $K$ are of this type; furthermore, even if a particular $K$ admits such a representation in theory, it may not be easy to find the explicit shape of $\mu$. In contrast, our approach works for \textit{all H\"older continuous kernels} without any restrictions on the shape and allows to get explicit approximations $\hat u_n$ of the optimal control $\hat u$. The lift procedure presented here is also different from the one used in \cite{Pham} (although they both are specific cases of the technique presented in \cite{CT}).

The lift used in the present paper was introduced in \cite{CT}, then generalized in \cite{CTMulti} for the multi-dimensional case,
but the approach itself can be traced back to \cite{CC1998}. It should be also emphasised that this method has its own limitations: in order to perform the lift, the kernel $K$ is required to have a specific representation of the form $K(t) = \langle g, e^{t\calA} \nu \rangle_{\mathbb H}$, $\t$, where $g$ and $\nu$ are elements of some Hilbert space $\mathbb H$ and $\{e^{t\calA},~t\in[0,T]\}$ is a uniformly continuous semigroup acting on $\mathbb H$ with $\calA \in \mathcal L(\mathbb H)$ and, in general, it may be hard to find feasible $\mathbb H$, $g$, $\nu$ and $\calA$. Here, we work with H\"older continuous kernels $K$ and we overcome this issue by approximating the kernel with Bernstein polynomials (which turn out to enjoy a simple representation of the required type). Then we solve the optimal control problem for the forward process with approximated kernel instead of the original one and we study convergence.

The paper is organised as follows. In section \ref{sec: lift}, we present our approach in case of a \emph{liftable} $K$ (i.e. $K$ having a representation in terms of $\mathbb H$, $g$, $\nu$ and $\calA$ mentioned above). Namely, we describe the lift procedure, give the necessary results from stochastic optimal control theory in Hilbert spaces as well as derive an explicit representation of the optimal control $\hat u$ by solving the associated Hamilton-Jacobi-Bellman equation. In section \ref{sec: approximation}, we introduce a liftable approximation for general H\"older continuous kernels, give convergence results for the solution to the approximated problem and discuss some numerical aspects for the latter. In section \ref{sec: simulations}, we illustrate the application of our technique with examples and simulations.


\section{Solution via Hilbert space-valued lift}\label{sec: lift}

\subsection{Preliminaries}

First of all, let us begin with some simple results on the optimization problem \eqref{eq: optimization problem}. Namely, we notice that $X_n^u$ and the optimization problem \eqref{eq: optimization problem} is well defined for any $u\in L^2_a$.
\begin{theorem}\label{th: appendix theorem} Let $K\in L^2([0,T])$. Then, for any $u \in L^2_a$,
    \begin{itemize}
        \item[1)] the forward Volterra Ornstein-Uhlenbeck-type equation \eqref{Forward_introduction} has a unique solution;
        \item[2)] there exists a constant $C>0$ such that
        \[
            \sup_{t\in[0,T]} \mathbb E[|X^u(t)|^2] \le C(1+ \lVert u\rVert^2_2),
        \]
        where $\lVert\cdot\rVert_2$ denotes the standard $L^2(\Omega\times[0,T])$ norm;
        \item[3)] $|J(u)| < \infty$.
    \end{itemize}
\end{theorem}
\begin{proof}
    Item 1) is evident since, for almost all $\omega\in\Omega$, the equation \eqref{Forward_introduction} treated pathwisely can be considered as a deterministic linear Volterra integral equation of the second kind that has a unique solution (see e.g. \cite{Tricomi_1985}). Next, it is straightforward to deduce that
    \begin{align*}
        \E \left[|X^u(t)|^2\right] &\le C\bigg(1+ \E\left[\left(\int_0^t K(t-s) u(s)ds\right)^2\right] + \E\left[\left(\int_0^t K(t-s) X^u(s)ds\right)^2\right] 
        \\
        &\qquad+ \E\left[\left(\int_0^t K(t-s) dW(s)\right)^2\right] \bigg)
        \\
        & \le C \left(1 +   \|K\|_2^2 \|u\|_2^2  +  \|K\|_2^2 \int_0^t \E\left[|X^u(s)|^2\right] ds + \|K\|_2^2\right)
        \\
        &\le C \left(1 +  \|u\|_2^2  +  \int_0^t \E\left[|X^u(s)|^2\right] ds\right).
    \end{align*}
    Now, item 2) follows from Gronwall's inequality. Finally, $\E [X^u(t)]$ satisfies the deterministic Volterra equation of the form
    \[
        \E[X^u(t)]= - \beta \int_0^t K(t-s)  \E[X^u(s)] ds + X(0) + \alpha\int_0^t K(t-s) \E[u(s)]ds 
    \]
    and hence can be represented in the form
    \begin{equation*}
    \begin{aligned}
        \E[X^u(t)] & = X(0) + \alpha \int_0^t K(t-s)\E[u(s)]ds - \beta \int_0^t R_\beta(t,s) X(0) ds \\
        &\quad - \alpha\beta\int_0^t R_\beta(t,s)\int_0^s K(s-v) \E[u(v)] dv ds
        \\
        & =: X(0) + \mathcal L u,
    \end{aligned}    
    \end{equation*}
    where $R_\beta$ is the resolvent of the corresponding Volterra integral equation and the operator $\mathcal L$ is linear and continuous. Hence $J(u)$ can be re-written as
    \begin{equation}\label{eq: LQ-form}
        J(u) =  -a_1\langle u,u \rangle_{L^2(\Omega\times [0,T]} + a_2 (X(0) + \mathcal L u),
    \end{equation}
    which immediately implies that $|J(u)| < \infty$. 
\end{proof}

\subsection{Construction of Markovian lift and formulation of the lifted problem} 

As anticipated above, in order to solve the optimization problem \eqref{eq: optimization problem} we will rewrite $X^u$ in terms of Markovian Hilbert space-valued process $\calZ^u$ using the lift presented in \cite{CT} and then apply the dynamic programming principle in Hilbert spaces. We start from the description of the core idea behind the Markovian lifts in case of liftable kernels. 

\begin{definition}
    Let $\bbH$ denote a separable Hilbert space with the scalar product $\langle \cdot, \cdot\rangle$. A kernel $K\in L^2([0,T])$ is called $\bbH$-\emph{liftable} if there exist $\nu, g\in \bbH$, $\lVert \nu \rVert_{\mathbb H} = 1$, and a uniformly continuous semigroup $\{e^{t\calA},~t\in[0,T]\}$ acting on  $\bbH$, $\calA \in \mathcal L(\mathbb H)$, such that
    \begin{equation}\label{Kernel_representation}
        K(t)=\langle g, e^{t\calA}\nu \rangle, \quad \t.
    \end{equation}
\end{definition}
For examples of liftable kernels, we refer to Section \ref{sec: simulations} and to \cite{CT}.

Consider a controlled Volterra Ornstein-Uhlenbeck process of the form \eqref{Forward_introduction} with a liftable kernel $K(t) = \langle g, e^{t\calA}\nu \rangle$, $\lVert \nu \rVert_{\mathbb H} = 1$, and denote $\zeta_0:=\frac{X(0)}{\|g\|_{\bbH}^2}g$ and
\[
    dV^u(t):=(\alpha u(t)-\beta X^u(t))dt+\sigma dW(t).
\]
Using the fact that $X(0)=\langle g, \zeta_0\rangle$, we can now rewrite \eqref{Forward_introduction} as follows:
\begin{align*}
	X^u(t) & = X(0)+\int_0^t K(t-s) dV^u(s) 
	\\
	&=\langle g, \zeta_0 \rangle + \int_0^t\langle g,e^{(t-s)\calA}\nu\rangle dV^u(s) 
	\\
	&=\left \langle g, \zeta_0+\int_0^t e^{(t-s)\calA} \nu dV^u(s) \right\rangle 
	\\
	&=:\langle g,\widetilde\calZ_t^u\rangle,
\end{align*}
where $\widetilde \calZ^u_t:=\zeta_0+\int_0^t e^{\calA t-s}\nu dV^u(s)$. It is easy to check that, $\widetilde\calZ^u$ is the unique solution of the infinite dimensional SDE
\begin{equation*}
    \widetilde\calZ_t^u = \zeta_0+ \int_0^t \Big(\calA(\widetilde\calZ_s^u-\zeta_0) + (\alpha u(s)-\beta\langle g,\widetilde\calZ_s^u\rangle) \nu \Big)ds + \int_0^t\sigma \nu dW(s)
\end{equation*}
and thus the process $\{\calZ_t^u, \t \}$ defined as $\calZ^u_t := \widetilde\calZ^u_t-\zeta_0$ satisfies the infinite dimensional SDE of the form
\begin{equation*}
     \calZ_t^u=\int_0^t\Big(\bar \calA\calZ_s^u-\nu\beta\langle g,\zeta_0\rangle+\nu \alpha u(s)\Big) ds +\int_0^t \sigma \nu dW(s),
\end{equation*}
where $\bar \calA$ is the linear bounded operator on $\mathbb H$ such that 
\begin{equation*}
    \bar\calA z:=\calA z-\beta\langle g, z \rangle \nu, \quad z\in\mathbb H.
\end{equation*}

These findings are summarized in the following theorem.

\begin{theorem}\label{th: H-valued representation}
    Let $\{X^u(t), \t\}$ be a Volterra Ornstein-Uhlenbeck process of the form \eqref{Forward_introduction} with the $\mathbb H$-liftable kernel $K(t) = \langle g, e^{t\calA}\nu \rangle$, $g,\nu\in\mathbb H$, $\lVert \nu \rVert_{\mathbb H} = 1$, $\calA\in\mathcal L(\mathbb H)$. Then, for any $\t$,
    \begin{equation}\label{eq: representation}
        X^u(t) = \langle g, \zeta_0 \rangle + \langle g, \calZ^u_t \rangle,
    \end{equation}
    where $\zeta_0:=\frac{X(0)}{\|g\|_{\bbH}^2}g$ and $\{\calZ^u_t,~\t\}$ is the $\mathbb H$-valued stochastic process given by
    \begin{equation}\label{Forward_lifted_1}
        \calZ_t^u=\int_0^t\Big(\bar \calA\calZ_s^u-\nu\beta\langle g,\zeta_0\rangle+\nu \alpha u(s)\Big) ds +\int_0^t \sigma \nu dW(s)
    \end{equation}
    and $\bar \calA \in \mathcal L(\mathbb H)$ is such that 
    \begin{equation*}
        \bar\calA z:=\calA z-\beta\langle g, z \rangle \nu, \quad z\in\mathbb H.
    \end{equation*}
\end{theorem}

Using Theorem \ref{th: H-valued representation}, one can rewrite the performance functional $J(u)$ from \eqref{eq: optimization problem} as 
\begin{equation}\label{performance_functional_extra_term}
    J^g (u) = \E\left[-\int_0^Ta_1u^2(s)ds+a_2\langle g,\calZ^u_T\rangle\right]+a_2\langle g,\zeta_0\rangle ,
\end{equation}
where the superscript $g$ in $J^g$ is used to highlight dependence on the $\mathbb H$-valued process $\calZ^u$. Clearly, maximizing \eqref{performance_functional_extra_term} is equivalent to maximizing 
\[
    J^g (u) - a_2\langle g,\zeta_0\rangle = \E\left[-\int_0^Ta_1u^2(s)ds+a_2\langle g,\calZ^u_T\rangle\right].
\]
Finally, for the sake of notation and coherence with literature, we will sometimes write our maximization problem as a minimization one by simply noticing that the maximization of the performance functional $J^g (u) - a_2\langle g,\zeta_0\rangle$ can be reformulated as the minimization of 
\begin{equation}\label{eq: optimization functional bar}
        \bar J^g(u) := - J^g (u) + a_2\langle g,\zeta_0\rangle = \E\left[\int_0^Ta_1u^2(s)ds - a_2\langle g,\calZ^u_T\rangle\right].
\end{equation}

In other words, in case of $\mathbb H$-liftable kernel $K$, the original optimal control problem \eqref{eq: optimization problem} can be replaced by the following one:
\begin{equation}\label{eq: lifted problem}
    \begin{cases}
        \calZ_t^u=\int_0^t\Big(\bar \calA\calZ_s^u-\nu\beta\langle g,\zeta_0\rangle+\nu \alpha u(s)\Big) ds +\int_0^t \sigma \nu dW(s),
        \\
        \bar J^g(u) := \E\left[\int_0^Ta_1u^2(s)ds - a_2\langle g,\calZ^u_T\rangle\right] \to \min,
    \end{cases} \quad u\in L^2_a. 
\end{equation}

\begin{remark}
    The machinery described above can also be generalized for strongly continuous semigroups on Banach spaces, see e.g. \cite{CTMulti,CT}. However, for our purposes it is sufficient to consider the case when $\calA$ is a linear bounded operator on a Hilbert space.
\end{remark}


\subsection{Solution to the lifted problem}

The optimal control problem \eqref{eq: lifted problem} completely fits the framework of dynamic programming principle stated in \cite[Chapter 6]{FGSw2017}. More precisely, consider the \textit{Hamilton-Jacobi-Bellman} (HJB) equation associated to the problem \eqref{eq: lifted problem} of the form
\begin{equation}\label{HJB no approx}
\begin{cases}
    \frac{\partial}{\partial t}v(t,z)&=-\frac 12 \text{Trace}\Big(\langle g,z\rangle \langle g,z\rangle \nabla^2v(t,z) \Big)-\langle \calA z -\beta\langle g,z\rangle\nu -\beta \langle g,\zeta_0\rangle\nu,\nabla v(t,z)\rangle
    \\
    &\quad-\calH(t,z,\nabla v(t,z)),
    \\
    v(T,z) &=-\langle a_2 g,z\rangle,
\end{cases}
\end{equation}
where by $\nabla v$ we denote the partial Gateaux derivative w.r.t. the spacial variable $z$ and the Hamiltonian functional $\calH: [0,T]\times\bbH^2 \to \R$ is defined as
\begin{equation*}
    \calH(t,z,\xi) := \inf_{u\in \mathbb R}\left\{a_1u^2+\alpha\langle\xi,\nu \rangle u \right\} = -\frac{\alpha^2 \langle \xi,\nu \rangle^2}{4 a_1}.
\end{equation*}
It is easy to check that the coefficients of the lifted forward equation \eqref{Forward_lifted_1} satisfy \cite[Hypothesis 6.8]{FGSw2017}, the Hamiltonian $\calH$ satisfies \cite[Hypothesis 6.22]{FGSw2017} and the term $-a_2 \langle g, \calZ^u_T \rangle$ in the performance functional \eqref{eq: optimization functional bar} satisfies \cite[Hypothesis 6.26]{FGSw2017}. Therefore, by \cite[Theorem 6.32]{FGSw2017} (see also \cite[Theorem 6.2]{FT}), \eqref{HJB no approx} has a unique mild solution $v$: $[0,T]\times \mathbb H \to \mathbb R$. Moreover, \cite[Theorem 6.35]{FGSw2017} implies that for any $u\in L^2_a$ 
\[
    \bar J^g (u) \ge v\left(0, 0\right)
\]
and the equality holds if and only if
\begin{equation}\label{eq: quadratic eq for u}
    a_1 u^2(t) + \alpha \left\langle \nabla v(t, \calZ^u_t) , \nu \right\rangle u(t) = - \frac{\alpha^2 \langle \nabla v(t, \calZ^u_t), \nu \rangle^2}{4a_1}, \quad t\in[0,T].
\end{equation}
Solving \eqref{eq: quadratic eq for u}, we obtain $\hat u(t)$, $\t$, which has the form
\begin{equation}\label{eq: shape of u-hat}
    \hat u(t) = - \frac{\alpha \langle \nabla v(t, \calZ^u_t), \nu \rangle}{2a_1}.
\end{equation}

\begin{remark}
    In general, \cite[Theorem 6.35]{FGSw2017} does not guarantee that $\hat u$ exists on the initial probability space, but instead considers the weak control framework, see \cite[Section 6]{FGSw2017} for more details. However, in our case optimal control exists in the strong sense and, as we will see later, $\hat u$ turns out to be deterministic. 
\end{remark}

Since the shape \eqref{eq: shape of u-hat} of the optimal control $\hat u$ depends on $\nabla v(t, \calZ^u_t)$, our next goal is to explicitly solve the HJB equation \eqref{HJB no approx}. The solution as well as the optimality statement are given in the next theorem.

\begin{theorem}\label{th: optimal control}
    1. The solution of the HJB equation \eqref{HJB no approx} associated with the lifted problem \eqref{eq: lifted problem} has the form
    \[
        v(t,z)=\langle w(t),z\rangle +c(t),
    \]
    where
    \begin{equation}\label{w_noepsilon}
        w(t) = - a_2 e^{-(t-T)\bar\calA^*} g , \quad \t,
    \end{equation}
    $\bar\calA^* = \calA^* - \beta \langle \nu, \cdot \rangle g$, and
    \begin{equation}\label{c_noepsilon}
        c(t) = -\int_t^T \left(\beta X(0) \langle w(s),\nu\rangle + \frac{\alpha^2}{4a_1}\langle w(s) ,\nu\rangle^2\right)ds, \quad \t.
    \end{equation}
    
    2. The solution $\hat u$ of the optimal control problem \eqref{eq: lifted problem} (and thus of the problem \eqref{eq: optimization problem}) has the form
    \begin{equation}\label{eq: optimal u}
        \hat u(t)=-\frac{\alpha}{2a_1}\langle w(t),\nu \rangle
        =\frac{\alpha a_2}{2a_1 }\langle g, e^{(T-t)\bar \calA}\nu\rangle,
    \end{equation}
    where $\bar\calA = \calA -\beta\langle g, \cdot \rangle \nu$.
\end{theorem}

\begin{proof}
    1. In order to solve the HJB equation \eqref{HJB no approx}, we will use the approach presented in \cite[Section 3.3]{GMS2009}. Namely, we will look for the solution in the form 
    \begin{equation}\label{shape v}
         v(t,z)=\langle w(t),z\rangle +c(t),
    \end{equation}
    where $w(t)$ and $c(t)$ are such that $\frac{\partial}{\partial t} v$ and $\nabla v$ are well-defined. In this case,
    \begin{align*}
      \frac{\partial}{\partial t} v(t,z)=\langle w'(t),z\rangle +c'(t),\quad \nabla v(t,z)=w(t), \quad \nabla^2v(t,z)=0, 
    \end{align*}
    and, recalling that $\langle g, \zeta_0 \rangle = X(0)$, we can rewrite the HJB equation \eqref{HJB no approx} as
    \begin{equation*}
    \begin{cases}
        \langle w'(t),z\rangle + \langle z , \bar\calA^* w(t)\rangle + c'(t) - \beta X(0) \langle w(t),\nu\rangle - \frac{\alpha^2}{4a_1}\langle w(t) ,\nu\rangle^2 =0
        \\    
        \langle w(T),z\rangle +c(T) =-\langle a_2 g,z\rangle.
    \end{cases}
    \end{equation*}
    Now it would be sufficient to find $w$ and $c$ that solve the following systems:
    \begin{equation}\label{two systems}
        \begin{cases}
            {\langle w'(t),z\rangle + \langle z , \bar\calA^* w(t)\rangle} = 0
            \\
            {\langle w(T),z\rangle +\langle a_2 g,z\rangle} = 0
        \end{cases}; \quad
        \begin{cases}
            {c'(t) - \beta X(0) \langle w(t),\nu\rangle - \frac{\alpha^2}{4a_1}\langle w(t) ,\nu\rangle^2} =0
            \\
            {c(T)} = 0,
        \end{cases}
    \end{equation}
    Noticing that the first system in \eqref{two systems} has to hold for all $z\in \mathbb H$, we can solve
    \[
        \begin{cases}
            w'(t) + \bar\calA^* w(t) = 0,
            \\
            w(T)+a_2 g = 0
        \end{cases}
    \]
    instead, which is a simple linear equation and its solution has the form \eqref{w_noepsilon}. Now it is easy to see that $c$ has the form $\eqref{c_noepsilon}$ and
    \[
        v(t,z) = \langle w(t), z \rangle + c(t), \quad \t.
    \]
    
    2. The result follows directly from the item 1 above and \cite[Theorem 6.35]{FGSw2017}.
\end{proof}

\begin{remark}
    By defining $$J(u,t, X(0)) := \E\left[-\int_t^Ta_1u^2(s)ds + a_2 X^u(T)\right],$$ and thus taking $J(u) := J(u,0, X(0))$ in \eqref{eq: optimization problem}, and consequently $\bar J^g(u) := \bar J^g(u,0, 0 )$ in \eqref{eq: optimization functional bar}, we obtain that Theorem \ref{th: optimal control} can be seen as a classical verification theorem, with statement $v(0,0) = \min_u \bar J^g(u,0,0).$
\end{remark}

\begin{remark}
     The approach described above can be extended by lifting to Banach space-valued stochastic processes. See \cite{DNG_Lift} for more details.
\end{remark}


\section{Approximate solution for forwards with H\"older kernels}\label{sec: approximation}

The crucial assumption in section \ref{sec: lift} that allowed to apply the optimization techniques in Hilbert space was the liftability of the kernel. However, in practice it is often hard to find a representation of the required type for the given kernel, and even if this representation is available, it is not always convenient from the implementation point of view. For this reason, we provide a liftable approximation for the Volterra Ornstein-Uhlenbeck process \eqref{Forward_introduction} for a general $C^h$-kernel $K$, where $C^h([0,T])$ denotes the set of $h$-Hölder continuous functions on $[0,T]$.

This section is structured as follows: first we approximate an arbitrary $C^h$-kernel by a liftable one in a uniform manner and introduce a new optimization problem where the forward dynamics is obtained from the original one  replacing the kernel $K$ with its liftable approximation. Afterwards, we prove that the optimal value of the approximated problem converges to the optimal value of the original problem and give an estimate for the rate of convergence. Finally, we discuss some numerical aspects that could be useful from the implementation point of view. 

\begin{remark}\label{rem: C is not important}
    In what follows, by $C$ we will denote any positive constant the particular value of which is not important and may vary from line to line (and even within one line). By $\|\cdot\|_2$ we will denote the standard $L^2(\Omega\times[0,T])$-norm.
\end{remark}


\subsection{Liftable approximation for Volterra Ornstein-Uhlenbeck processes with H\"older continuous kernels}

Let $K\in C([0,T])$, $\bbH=L^2(\R)$, the operator $\calA$ be the 1-shift operator acting on $\bbH$, i.e.
\[
    (\calA f)(x) = f(x+1), \quad f\in \bbH,
\]
and denote $K_n$ a Bernstein polynomial approximation for $K$ of order $n\ge 0$, i.e.
\begin{equation}\label{eq: liftable approximation}
\begin{aligned}
    K_n(t) &= \frac{1}{T^n}\sum_{k=0}^n K\left(\frac{Tk}{n}\right) \binom{n}{k}t^k(T-t)^{n-k} 
    \\
    &=: \sum_{k=0}^n \kappa_{n,k} t^k, \quad t\in [0,T],
\end{aligned}    
\end{equation}
where
\begin{equation}\label{eq: choice of kappas}
    \kappa_{n,k} := \frac{1}{T^k}\sum_{i=0}^k (-1)^{k-i} K\left(\frac{iT}{n}\right)\binom{n}{i}\binom{n-i}{k-i}.
\end{equation}
Observe that 
\begin{align*}
    (e^{t\calA} \mathbbm 1_{[0,1]})(x) = \sum_{k=0}^\infty \frac{t^k}{k!} \left[\calA ^k \mathbbm 1_{[0,1]}\right](x) = \sum_{k=0}^\infty \frac{t^k}{k!} \mathbbm 1_{[-k,-k+1]}(x)
\end{align*}
and hence $K_n$ is \textit{$\bbH$-liftable} as
\[
    K_n (t) = \left\langle g_n, e^{t\calA} \nu \right\rangle_{\bbH} = \sum_{k=0}^n \kappa_{n,k} t^k, \quad t\in[0,T],
\]
with $g_n := \sum_{k=0}^n k! \kappa_{n,k} \mathbbm 1_{[-k,-k+1]}$ and $\nu := \mathbbm 1_{[0,1]}$.

By the well-known approximating property of Bernstein polynomials, for any $\varepsilon > 0$,  there exist $n = n(\varepsilon) \in \mathbb N_0$ such that  
\[
    \sup_{t\in[0,T]} \left| K(t) - K_n(t) \right| < \varepsilon.
\]
Moreover, if additionally $K \in C^h([0,T])$ for some $h\in(0,1)$, \cite[Theorem 1]{Bernstein} guarantees that for all $\t$
\begin{equation}\label{eq: order of Bernstein approximation}
    |K(t)-K_n(t)|\leq H\left(\frac{t(T-t)}{n}\right)^{h/2} \le \frac{H T^h}{2^h} n^{-\frac{h}{2}}, 
\end{equation}
where $H>0$ is such that
\begin{equation}\label{eq: Holder cond for K in remark}
    |K(t)-K(s)|\leq H |t-s|^h, \quad s,\t.
\end{equation}

Now, consider a controlled Volterra Ornstein-Uhlenbeck process $\{X^u(t),~\t\}$ of the form \eqref{Forward_introduction} with the kernel $K \in C^h([0,T])$ satisfying \eqref{eq: Holder cond for K in remark}. For a given admissible $u$ define also a stochastic process $\{X^u_n (t),~t\in [0,T]\}$ as a solution to the stochastic Volterra integral equation of the form
\begin{equation}\label{X_epsilon}
        X_n^u(t)=X(0)+\int_0^t K_n(t-s) \Big(\alpha u(s)-\beta X_n^u(s)\Big)ds+\sigma\int_0^t K_n(t-s) dW(s), \quad t\in[0,T],
\end{equation}
where $K_n (t) = \sum_{k=0}^n \kappa_{n,k} t^k$ with $\kappa_{n,k}$ defined by \eqref{eq: choice of kappas}, i.e. the Bernstein polynomial approximation of $K$ of degree $n$. 

\begin{remark}\label{rem: Holder continuity of the noise}
      It follows from \cite[Corollary 4]{Holder} that both stochastic processes $\int_0^t K(t-s)dW(s)$ and $\int_0^t K_n(t-s)dW(s)$, $\t$, have modifications that are Hölder continuous at least up to the order $h \wedge \frac 1 2$. From now on, these modifications will be used.
\end{remark}

Now we move to the main result of this subsection.

\begin{theorem}\label{th: conv of Xeps}
    Let $K\in C^h([0,T])$, $u \in L^2_a$, and $X^u$, $X^u_n$ are given by \eqref{Forward_introduction} and \eqref{X_epsilon} respectively. Then there exists $C>0$ which does not depend on $n$ or $u$ such that for any admissible $u\in L^2_a$:
    \begin{equation*}
        \sup_{t\in[0,T]} \E\left[|X^u(t) -X^u_{n}(t)|^2\right]\leq C (1+\|u\|^2_2) n^{-h}.
    \end{equation*}
\end{theorem}

\begin{proof}
    First, by Theorem \ref{th: appendix theorem}, there exists a constant $C>0$ such that
    \begin{equation}\label{eq: bound on second moment for X}
        \sup_{\t}\E[|X^u(t)|^2]\leq C(1+ \|u\|_2^2).
    \end{equation}
    Consider an arbitrary $\tau \in [0,T]$, and denote $\Delta(\tau):=\sup_{t\in[0,\tau]} \E\left[|X^u(t) -X^u_{n}(t)|^2\right]$. Then
    \begin{align*}
        \Delta(\tau)&=\sup_{t\in[0,\tau]}\E\Bigg[\Bigg|\int_0^t K(t-s)\Big(\alpha u(s)-\beta X^u(s)\Big)ds +\int_0^t K_n(t-s)\Big(\alpha u(s)-\beta X_n^u(s)\Big)ds 
        \\
        &\quad +\int_0^t \sigma \Big( K(t-s)-K_n(t-s)\Big)dW(s)\Bigg|^2\Bigg]
        \\
        &\leq C \sup_{t\in[0,\tau]}\E\Bigg[ \int_0^t\Bigg|   \Big( K(t-s)-K_n (t-s) \Big)u(s)  \Bigg|^2 ds\Bigg]
        \\
        &\quad+C\sup_{t\in[0,\tau] } \E\Bigg[ \int_0^t \Bigg|K_n(t-s)\Big(X^u(s)-X_n^u(s)\Big)\Bigg|^2 ds\Bigg]
        \\
        &\quad+ C\sup_{t\in[0,\tau]}\E\Bigg[ \int_0^t\Bigg| X^u(s) \Big( K(t-s)-K_n(t-s)\Big)\Bigg|^2 ds\Bigg]
        \\
        &\quad +C  \sup_{t\in[0,\tau]}\E\left[\Bigg|\int_0^t \Big(K(t-s)-K_n(t-s)\Big) dW(s)\Bigg|^2  \right].
    \end{align*}
    Note that, by \eqref{eq: order of Bernstein approximation} we have that
    \begin{align*}
        \sup_{t\in[0,\tau]}\E\Bigg[ \int_0^t\Bigg|   \Big( K(t-s)-K_n (t-s) \Big)u(s)  \Bigg|^2 ds\Bigg] & \le Cn^{-h} \lVert u \rVert^2_2.
    \end{align*}
    Moreover, since $\{K_n,~n\ge 1\}$ are uniformly bounded due to their uniform convergence to $K$ it is true that
    \begin{align*}
        \sup_{t\in[0,\tau] } \E\Bigg[ \int_0^t \Bigg|K_n(t-s)\Big(X^u(s)-X_n^u(s)\Big)\Bigg|^2 ds\Bigg] & \le C \int_0^\tau \Delta(s)ds
    \end{align*}
    with $C$ not dependent on $n$, and from \eqref{eq: order of Bernstein approximation}, \eqref{eq: bound on second moment for X} one can deduce that
    \begin{align*}
        \sup_{t\in[0,\tau]}\E\Bigg[ \int_0^t\Bigg| X^u(s) \Big( K(t-s)-K_n(t-s)\Big)\Bigg|^2 ds\Bigg] \le C n^{-h} (1+ \lVert u \rVert^2_2).
    \end{align*}
    Lastly, by the Ito isometry and \eqref{eq: order of Bernstein approximation},
    \[
        \sup_{t\in[0,\tau]}\E\left[\Bigg|\int_0^t \Big(K(t-s)-K_n(t-s)\Big) dW(s)\Bigg|^2  \right] \le C n^{-h}.
    \]
    Hence
    \[
        \Delta(\tau) \le Cn^{-h} (1+\|u\|_2^2)+  C \int_0^t \Delta(s) ds,
    \]
    where $C$ is a positive constant (recall that it may vary from line to line). The final result follows from Gronwall's inequality.
\end{proof}


\subsection{Liftable approximation of the optimal control problem}

As it was noted before, our aim is to find an approximate solution to the the optimization problem \eqref{eq: optimization problem} by solving the liftable problem of the form
\begin{equation}\label{eq: approximated optimization problem}
    \begin{cases}
        X_n^u(t)=X(0)+\int_0^t K_n(t-s)\Big(\alpha u(s)-\beta X_n^u(s)\Big)ds+\sigma\int_0^t K_n(t-s) dW(s),
        \\
        J_n(u) := \E\left[-\int_0^Ta_1 u^2(s) ds+a_2X_n^u (T) \right] \to \max,
    \end{cases}
\end{equation}
where the maximization is performed over $u\in L^2_a$. In \eqref{eq: approximated optimization problem}, $K_n$ is the Bernstein polynomial approximation of $K\in C^h([0,T])$, i.e. 
\[
    K_n(t) = \langle g_n, e^{t\calA} \nu\rangle, \quad \t,
\]
where $\calA \in \mathcal{L} \left( \bbH\right)$ acts as $(\calA f)(x+1)$, $\nu = \mathbbm 1_{[0,1]}$ and $ g_n = \sum_{k=0}^n k! \kappa_{n,k} \mathbbm 1_{[-k,-k+1]}$ with $\kappa_{n,k}$ defined by \eqref{eq: choice of kappas}. Due to the liftability of $K_n$, the problem \eqref{eq: approximated optimization problem} falls in the framework of section \ref{sec: lift}, so, by Theorem \ref{th: optimal control}, the optimal control $\hat u_n$ has the form \eqref{eq: optimal u}:
\begin{equation}\label{eq: form of hat un}
    \hat u_n(t) = \frac{\alpha a_2}{2a_1 }\langle g_n, e^{(T-t)\bar \calA_n} \nu \rangle, \quad \t,
\end{equation}
where $\bar\calA_n := \calA -\beta\langle g_n, \cdot \rangle \nu$. The goal of this subsection is to prove the convergence of the optimal performance in the approximated dynamics to the actual optimal, i.e. 
\[
    J_n(\hat u_n) \to \sup_{u\in L^2_a} J(u), \quad n \to \infty,
\]
where $J$ is the performance functional from the original optimal control problem \eqref{eq: optimization problem}. 
\begin{proposition}\label{prop: u cannot be big}
    Let the kernel $K \in C^h([0,T])$. Then
    \begin{equation}\label{propeq: J conv 1}
        \sup_{n\in \mathbb N} J_n (u) \to - \infty \qquad\text{as } \|u\|_2 \to \infty,
    \end{equation}
    \begin{equation}\label{propeq: J conv 2}
        J(u) \to - \infty \qquad \text{as } \|u\|_2 \to \infty,
    \end{equation}
    where $\|\cdot\|_2$ denotes the standard $L^2(\Omega\times[0,T])$ norm.
\end{proposition}

\begin{proof}
    We prove only \eqref{propeq: J conv 1}; the proof of \eqref{propeq: J conv 2} is the same. Let $u\in L^2_a$ be fixed. For any $n\in \mathbb N$ denote
    \[
        G_n(t) := \int_0^t K_n(t-s) dW(s), \quad \t,
    \]
    and notice that for any $\t$ we have that
    \begin{align*}
        |X_n^u(t)| &\le X(0) + \alpha \int_0^t |K_n (t-s)| |u(s)| ds + \beta \int_0^t  |K_n (t-s)||X^u_n (s)| ds  + \sigma \left| G_n(t)\right|
        \\
        &\le C\left(1 + \left(\int_0^T u^2(s) ds\right)^{\frac{1}{2}} + \int_0^t |X_n^u(s)|ds + \sup_{r\in[0,T]}\left| G_n(r)\right|\right),
    \end{align*}
    where $C > 0$ is a deterministic constant that does not depend on $n$, $t$ or $u$ (here we used the fact that $K_n \to K$ uniformly on $[0,T]$). Whence, for any $n\in \mathbb N$,
    \begin{align}\label{eq: pre-Gronwall}
        \mathbb E \left[ |X_n^u(t)| \right] \le C\left( 1 + \|u\|_2 + \int_0^t \E \left[|X_n^u(s)|\right]ds + \mathbb E\left[\sup_{r\in[0,T]}\left| G_n(r) \right|\right]\right).
    \end{align}
    Now, let us prove that there exists a constant $C>0$ such that
    \[
        \sup_{ n\in \mathbb N } \mathbb E\left[\sup_{r\in[0,T]}\left| G_n(r) \right|\right] < C.
    \]
    First note that, by Remark \ref{rem: Holder continuity of the noise}, for each $n\in \mathbb N$ and $\delta \in \left(0, \frac{h}{2} \wedge \frac{1}{4}\right)$ there exists a random variable $\Upsilon_n = \Upsilon_n (\delta)$ such that
    \[
        \left|G_n(r_1) - G_n(r_2)\right| \le \Upsilon_n |r_1 - r_2|^{h\wedge \frac 12 - 2\delta}
    \]
    and whence
    \[
        \sup_{r\in[0,T]}\left| G_n(r) \right| \le T^{h\wedge \frac 12 - 2\delta} \Upsilon_n.
    \]
    Thus it is sufficient to check that $\sup_{n\in \mathbb N}\E \Upsilon_n < \infty$. It is known from \cite{Holder} that one can put
    \[
        \Upsilon_n := C_\delta \left( \int_0^T \int_0^T \frac{|G_n (x) - G_n (y)|^{p}}{|x-y|^{(h\wedge \frac{1}{2} - \delta)p+1}}dxdy \right)^{\frac{1}{p}},
    \]
    where $p:=\frac{1}{\delta}$ and $C_\delta > 0$ is a constant that does not depend on $n$. Let $p' > p$. Then Minkowski integral inequality yields
    \begin{equation}\label{eq: estim for Upsilon}
    \begin{aligned}
        \left(\E \Upsilon^{p'}_n \right)^{\frac{p}{p'}} &= C_\delta^{p} \left( \E\left[\left( \int_0^T \int_0^T \frac{|G_n (x) - G_n (y)|^{p}}{|x-y|^{(h\wedge \frac{1}{2} - \delta)p+1}}dxdy \right)^{\frac{p'}{p}}\right] \right)^{\frac{p}{p'}}
        \\
        & \le C_\delta^{p} \int_0^T \int_0^T \frac{\left(\E\left[ |G_n (x) - G_n (y)|^{p'}\right]\right)^{\frac{p}{p'}}}{|x-y|^{(h\wedge \frac{1}{2} - \delta)p+1}}dxdy.
    \end{aligned}
    \end{equation}
    Note that, by \cite[Proposition 2]{Bernstein}, every Bernstein polynomial $K_n$ that corresponds to $K$ is Hölder continuous of the same order $h$ and with the same constant $H$, i.e.
    \begin{equation*}
        |K_n(t)-K_n(s)|\leq H|t-s|^h, \quad s,\t,
    \end{equation*}
    whenever
    \[
        |K(t)-K(s)|\leq H|t-s|^h, \quad s,\t.
    \]
    This implies that there exists a constant $C$ which does not depend on $n$ such that 
    \begin{align*}
        \E\left[ |G_n (x) - G_n (y)|^{p'}\right] &= C \left( \int_0^{x\wedge y} (K_n(x-s) - K_n(y-s))^2ds + \int^{x\vee y}_{x\wedge y} K^2_n(x \vee y-s)ds \right)^{\frac{p'}{2}}
        \\
        &\le C|x-y|^{p'(h\wedge \frac{1}{2})}.
    \end{align*}
    Plugging the bound above to \eqref{eq: estim for Upsilon}, we get that
    \begin{align*}
        \left(\E\left[ \Upsilon^{p'}_n\right] \right)^{\frac{p}{p'}} &\le  C \int_0^T \int_0^T |x-y|^{(h\wedge \frac 12) p - (h\wedge \frac{1}{2} -\delta)p-1}dxdy
        \\
        &= C \int_0^T \int_0^T |x-y|^{ -1 + \delta p}dxdy 
        \\
        & < C,
    \end{align*}
    where $C>0$ denotes, as always, a deterministic constant that does not depend on $n$, $t$, $u$ and may vary from line to line. 
    
    Therefore, there exists a constant, again denoted by $C$ not depending on $n$, $t$ or $u$ such that
    \[
        \sup_{n\in \mathbb N} \mathbb E\left[ \Upsilon_{n} \right] < C
    \]
    and thus, by \eqref{eq: pre-Gronwall},
    \[
        \mathbb E \left[ |X_n^u(t)| \right] \le C\left( 1 + \|u\|_2 + \int_0^t \E\left[ |X_n^u(s)|\right]ds\right).
    \]
    By Gronwall's inequality, there exists $C>0$ which does not depend on $n$ such that
    \[
        \mathbb E \left[ |X_n^u(T)| \right] \le C(1+\|u\|_2),
    \]
    and so 
    \[
        \sup_{n\in \mathbb N} J_n (u) \le C(1+ \|u\|_2) - \|u\|_2^2 \to -\infty, \quad \|u\|_2\to\infty.
    \]
\end{proof}

\begin{theorem}
    Let $K \in C^h([0,T])$ and $K_n$ be its Bernstein polynomial approximation of order $n$. Then there exists constant $C>0$ such that
    \begin{equation}\label{eq: rate of convergence to the optimal}
        \Big|J_n(\hat u_n) - \sup_{u\in L^2_a} J(u)\Big| \le C n^{-\frac{h}{2}}.
    \end{equation}
    Moreover, $\hat u_n$ is ``\emph{almost optimal}'' for $J$ in the sense that there exists a constant $C>0$ such that
    \[
        \Big|J(\hat u_n) - \sup_{u\in L^2_a} J(u)\Big| \le C n^{-\frac{h}{2}}.
    \]  
\end{theorem}
\begin{proof}
    First, note that for any $r \ge 0$
    \begin{equation}\label{eq: uniform convergence over bounded sets}
        \sup_{u\in B_r} \Big|J_n(u) - J(u)\Big| \le C (1+r^2)^{\frac 1 2} n^{-\frac{h}{2}},
    \end{equation}
    where $B_r := \{u\in L^2_a:~\|u\|_2 \le r\}$. Indeed, by definitions of $J$, $J_n$ and Theorem \ref{th: conv of Xeps}, for any $u\in B_r$:
    \begin{equation}\label{eq: estim of Jeps-J}
    \begin{aligned}
        \Big|J_n(u) - J(u)\Big| & =\Big|\E[X_n^u(T)-X^u(T)]\Big| \le C(1 + \lVert u \rVert^2_2)^{\frac{1}{2}} n^{-\frac{h}{2}}
        \\
        &\le C (1+r^2)^{\frac 1 2} n^{-\frac{h}{2}}.
    \end{aligned}
    \end{equation}
    
    In particular, this implies that there exists $C>0$ that does not depend on $n$ such that $J(0) - C < J_n(0)$, so, by Proposition \ref{prop: u cannot be big}, there exists  $r_0>0$ that does not depend on $n$ such that $\|u\|_2 > r_0$ implies
    \[
        J_n(u) < J(0) - C < J_n(0), \quad n\in\mathbb N.
    \]
    In other words, all optimal controls $\hat u_n$, $n\in\mathbb N$ must be in the ball $B_{r_0}$ and that $\sup_{u\in L^2_a} J(u) = \sup_{u\in B_{r_0}} J(u)$. This, together with uniform convergence of $J_n$ to $J$ over bounded subsets of $L^2_a$ and estimate \eqref{eq: uniform convergence over bounded sets}, implies that there exists $C>0$ not dependent on $n$ such that
    \begin{equation}\label{eq: uniform convergence approx}
        \Big|J_n(\hat u_n) - \sup_{u\in L^2_a} J(u)\Big| \le C n^{-\frac{h}{2}}.
    \end{equation}
    Finally, taking into account \eqref{eq: uniform convergence over bounded sets} and \eqref{eq: uniform convergence approx} as well as the definition of $B_{r_0}$,
    \begin{align*}
        \Big|J(\hat u_n) - \sup_{u\in L^2_a} J(u)\Big| & \le \Big|J(\hat u_n) - J_n(\hat u_n)\Big| + \Big|J_n(\hat u_n) - \sup_{u\in L^2_a} J(u)\Big|
        \\
        &  \le \Big|J(\hat u_n) - J_n(\hat u_n)\Big| + \Big|J_n(\hat u_n) - \sup_{u\in B_{r_0}} J(u)\Big|
        \\
        &\le C n^{-\frac{h}{2}}.
    \end{align*}
    which ends the proof.
\end{proof}

\begin{theorem} 
    Let $K \in C^h([0,T])$ and $\hat u_n$ be defined by \eqref{eq: form of hat un}. Then the optimization problem \eqref{eq: optimization problem} has a unique solution $\hat u \in L^2_a$ and
    \[
        \hat u_n \to \hat u, \quad n\to\infty,
    \]
    in the weak topology of $L^2(\Omega\times[0,T])$.
\end{theorem}
\begin{proof}
    By \eqref{eq: LQ-form}, the performance functional $J$ can be represented in a linear-quadratic form as 
    \[
        J(u) =  -a_1\langle u,u \rangle_{L^2(\Omega\times [0,T]} + a_2 (X(0) + \mathcal L u),
    \]
    where $\mathcal L$: $L^2(\Omega\times[0,T]) \to L^2(\Omega\times[0,T])$ is a continuous linear operator. Then, by \cite[Theorem 9.2.6]{Allaire_2007}, there exists a unique $\hat u \in L^2(\Omega\times[0,T])$ that maximizes $J$ and, moreover, $\hat u_n \to \hat u$ weakly as $n\to\infty$. Furthermore, since all $\hat u_n$ are deterministic, so is $\hat u$; in particular, it is adapted to filtration generated by $W$ which implies that $\hat u \in L^2_a$. 
\end{proof}


\subsection{Algorithm for computing $\hat u_n$}\label{subsec: M}
    
The explicit form of $\hat u_n$ given by \eqref{eq: form of hat un} is not very convenient from the implementation point of view since one has to compute $e^{(T-t) \bar \calA_n}\nu=e^{(T-t)\bar \calA_n}\mathbbm{1}_{[0,1]}$, where $\bar\calA_n := \calA -\beta\langle g_n, \cdot \rangle \mathbbm 1_{[0,1]}$, $(\calA f)(x) = f(x+1)$. A natural way to simplify the problem is to truncate the series 
\begin{equation*}
    \sum_{k=0}^\infty \frac{(T-t)^k}{k!}\bar \calA_{n}^k\mathbbm{1}_{[0,1]} \approx \sum_{k=0}^M \frac{(T-t)^k}{k!}\bar \calA_{n}^k\mathbbm{1}_{[0,1]}
\end{equation*}
for some $M\in\mathbb N$. However, even after replacing $e^{(T-t)\bar\calA_n}$ in \eqref{eq: form of hat un} with its truncated version, we still need to be able to compute $\bar \calA_{n}^k\mathbbm{1}_{[0,1]}$ for the given $k \in \mathbb N $. An algorithm to do so is presented in the proposition below.
\begin{proposition}\label{prop: gammas}
    For any $k\in \mathbb N \cup\{0\}$, 
    \begin{equation*}
        \bar \calA_n^k\mathbbm{1}_{[0,1]}=\sum_{i=0}^k\gamma({i,k})\mathbbm{1}_{[-i,-i+1]},
    \end{equation*}
    where, $\gamma({0,0})=1$ and, for all $k\geq 1$, 
    \begin{equation*}
        \gamma({i,k})=
        \begin{cases}
             \gamma({i-1,k-1}), & i=1,...,k\\
             \sum_{j=0}^{(k-1)\wedge n}(-\beta)j! \kappa_{n,j}\gamma({j,k-1}), & i=0.
        \end{cases}
    \end{equation*}
\end{proposition}
\begin{proof}
    The proof follows an inductive argument. The statement for $\gamma({0,0})$ is obvious. Now let
    \begin{equation*}
        \bar \calA_n^{k-1}\mathbbm{1}_{[0,1]}=\sum_{i=0}^{k-1}\gamma({i,k-1})\mathbbm{1}_{[-i,-i+1]}.
    \end{equation*}
    Then 
    \begin{align*}
        \bar \calA_n^{k}\mathbbm{1}_{[0,1]}&=\bar\calA_n\Big(\bar \calA_n^{k-1}\mathbbm{1}_{[0,1]}\Big)
            \\
        &= \sum_{i=0}^{k-1}\gamma({i,k-1})\bar\calA_n\mathbbm{1}_{[-i,-i+1]}
            \\
        &=\sum_{i=1}^{k}\gamma({i-1,k-1})\mathbbm{1}_{[-i,-i+1]}
        \\
        &\quad+\mathbbm{1}_{[0,1]}(-\beta)\left \langle \sum_{j=0}^{k-1}\gamma({j,k-1})\mathbbm{1}_{[-j,-j+1]}, \sum_{j=0}^n j! \kappa_{n,j}\mathbbm{1}_{[-j,-j+1]} \right\rangle 
            \\
        &=\sum_{i=1}^{k}\gamma({i-1,k-1})\mathbbm{1}_{[-i,-i+1]}+\mathbbm{1}_{[0,1]}\sum_{j=0}^{(k-1)\wedge n}(-\beta)j! \kappa_{n,j}\gamma({j,k-1}).
    \end{align*}
\end{proof}

Finally, consider
\begin{align}
    \hat u_{n,M}(t) &:= \frac{\alpha a_2}{2a_1 }\left\langle g_n, \sum_{k=0}^M \frac{(T-t)^k}{k!}\bar \calA_{n}^k\mathbbm{1}_{[0,1]} \right\rangle \nonumber
    \\
    & = \frac{\alpha a_2}{2a_1 }\left\langle \sum_{i=0}^n i! \kappa_{n,i} \mathbbm 1_{[-i,-i+1]} , \sum_{k=0}^M \sum_{i=0}^k \frac{(T-t)^k}{k!}\gamma({i,k})\mathbbm{1}_{[-i,-i+1]} \right\rangle \nonumber
    \\
    & = \frac{\alpha a_2}{2a_1 }\left\langle \sum_{i=0}^n i! \kappa_{n,i} \mathbbm 1_{[-i,-i+1]} , \sum_{i=0}^M \left(\sum_{k=i}^M \frac{(T-t)^k}{k!}\gamma({i,k})\right) \mathbbm{1}_{[-i,-i+1]} \right\rangle \nonumber
    \\
    &= \frac{\alpha a_2}{2a_1 }\sum_{i=0}^{n \wedge M} \sum_{k=i}^M  \frac{i! \kappa_{n,i}\gamma({i,k})}{k!}(T-t)^k \nonumber 
    \\
    &= \frac{\alpha a_2}{2a_1 }\sum_{k=0}^M \left(\sum_{i=0}^{k\wedge n}   \frac{i! \kappa_{n,i}\gamma({i,k})}{k!}\right)(T-t)^k, \label{eq: uMn}
\end{align}
where $\kappa_{n,i}$ are defined by \eqref{eq: choice of kappas} and $\gamma({i,k})$ are from Proposition \ref{prop: gammas}.

\begin{theorem}
    Let $n\in\mathbb N$ be fixed and $M \ge (T-t)\lVert \bar\calA_n\rVert_{\mathcal L}$, where $\lVert\cdot\rVert_{\mathcal L}$ denotes the operator norm. Then, for all $\t$,
    \[
        |\hat u_n (t) - \hat u_{n,M}(t)| \le \frac{\alpha a_2}{2a_1 } \lVert g_n \rVert e^{ (T-t)\|\bar\calA_n\|_{\mathcal L} } \left(1 - e^{ - \frac{(T-t)\|\bar\calA_n\|_{\mathcal L}}{M+1} }\right).
    \]
    Moreover,
    \[
        \sup_{\t}|\hat u_n (t) - \hat u_{n,M}(t)| \le \frac{\alpha a_2}{2a_1 } \lVert g_n \rVert e^{ T\|\bar\calA_n\|_{\mathcal L} } \left(1 - e^{ - \frac{T\|\bar\calA_n\|_{\mathcal L}}{M+1} }\right) \to 0, \quad M \to\infty.
    \]
\end{theorem}
\begin{proof}
    One has to prove the first inequality and the second one then follows. It is clear that
    \[
        |\hat u_n (t) - \hat u_{n,M}(t)| \le \frac{\alpha a_2}{2a_1 } \lVert g_n \rVert \left \|\sum_{k=M+1}^\infty \frac{(T-t)^k}{k!}\bar\calA_n^k \mathbbm 1_{[0,1]}\right \|
    \]
    and, if $M \ge (T-t)\lVert \bar\calA_n\rVert_{\mathcal L}$, we have that
    \begin{align*}
        \left \|\sum_{k=M+1}^\infty \frac{(T-t)^k}{k!}\bar\calA_n^k \mathbbm 1_{[0,1]}\right \| &\le \sum_{k=M+1}^\infty \frac{\left((T-t)\left\|\bar\calA_n\right\|_{\mathcal L}\right)^k}{k!}
        \\
        & \le e^{ (T-t)\|\bar\calA_n\|_{\mathcal L} } \left(1 - e^{ - \frac{(T-t)\|\bar\calA_n\|_{\mathcal L}}{M+1} }\right),
    \end{align*}
    where we used a well-known result on tail probabilities of Poisson distribution (see e.g. \cite{Samuels1965}).
\end{proof}


\section{Examples and simulations}\label{sec: simulations}

\begin{example}[\emph{monomial kernel}]\label{ex: example of optimization 1}
    Let $N \in \mathbb N$ be fixed. Consider an optimization problem of the form
    \begin{equation}\label{eq: example of optimization 1}
        \begin{cases}
            X^u(t) = X(0)+\int_0^t(t-s)^{N}\Big(u(s)-X^u(s)\Big)ds+\int_0^t(t-s)^{N} dW(s),
            \\
            \E\left[X^u(T) - \int_0^Tu^2(s)ds\right] \to \max,
        \end{cases}
    \end{equation}
    where, as always, we optimize over $u\in L^2_a$. The kernel $K(t) = t^{N}$ is $\bbH$-liftable, 
    \begin{equation*}
        t^{N}=\langle {N}! \mathbbm{1}_{[-N,-N+1]},e^{ t\calA}\mathbbm{1}_{[0,1]}\rangle,
    \end{equation*}
    where $(\calA f)(x) = f(x+1)$, $f\in \bbH$. By Theorem \ref{th: optimal control}, the optimal control for the problem \eqref{eq: example of optimization 1} has the form
    \[
        \hat u(t) = \frac{N!}{2} \langle  \mathbbm{1}_{[-N,-N+1]} , e^{(T-t)\bar \calA}\mathbbm{1}_{[0,1]} \rangle ,
    \]
    where $\bar\calA =\calA-N!\langle \mathbbm{1}_{[-N,-N+1]}, \cdot \rangle \mathbbm{1}_{[0,1]} $. In this simple case, we are able to find an explicit expression for $e^{(T-t)\bar \calA^*} \mathds{1}_{[-i,-i+1]}$. Indeed, it is easy to see that, for any $i\in \mathbb N \cup \{0\}$, $p \in \mathbb N\cup \{0\}$ and $q=0,1,...,N$,
    \[
        \bar\calA^{p(N+1)+q} \mathbbm 1_{[0,1]} = \sum_{j=0}^p (-1)^{p-j} (N!)^{p-j} \mathbbm 1_{[-j(N+1) - q, -j(N+1) - q + 1]}
    \]
    and whence
    \begin{align*}
        \langle  \mathbbm{1}_{[-N,-N+1]},  &e^{(T-t)\bar\calA}\mathds{1}_{[0,1]} \rangle 
        \\
        & = \left\langle \mathbbm{1}_{[-N,-N+1]}, \sum_{p=0}^\infty \sum_{q=0}^{N} \frac{(T-t)^{pN + p + q}}{(pN+p+q)!} \sum_{j=0}^p (-1)^{p-j} (N!)^{p-j} \mathbbm 1_{[-j(n+1) - q, -j(N+1) - q + 1]}   \right\rangle 
        \\
        & = \sum_{p=0}^\infty \frac{(T-t)^{pN + p + N}}{(pN+p+N)!} (-1)^{p} (N!)^{p} 
        \\
        & = (T-t)^{N} E_{N+1, N+1}(-N!(T-t)^{N+1}),
    \end{align*}
    where $E_{a,b}(z) := \sum_{p=0}^\infty \frac{z^p}{\Gamma(ap+b)}$ is the Mittag-Leffler function. This, in turn, implies that
    \begin{equation}\label{eq: explicit u for monomial kernel}
        \hat u(t) = \frac{N!(T-t)^{N}}{2} E_{N+1, N+1}(-N!(T-t)^{N+1}).
    \end{equation}
    
    On Fig. \ref{Fig: convergence}, the black curve depicts the optimal $\hat u$ computed for the problem \ref{eq: example of optimization 1} with $K(t) = t^2$ and $T=2$ using \eqref{eq: explicit u for monomial kernel}; the othere curves are the approximated optimal controls $\hat u_{n,M}$ (as in \eqref{eq: uMn}) computed for $n=1,2,5,10$ and $M=20$.
    \begin{figure}[h!]
        \centering
        \includegraphics[height=0.4\textwidth]{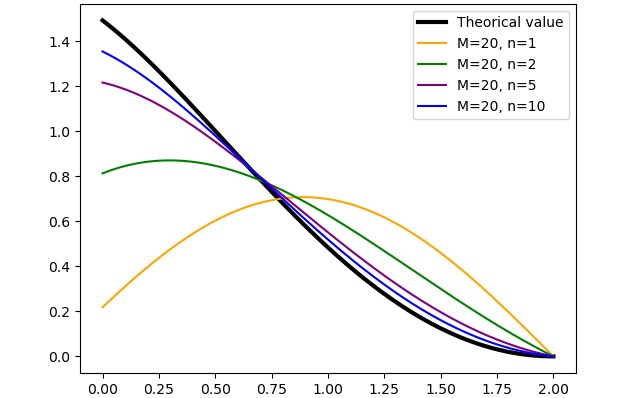}
        \caption{Optimal control of Volterra Ornstein-Uhlenbeck process with monomial kernel $K(t) = t^2$ (in black) and control approximants $\hat u_{n,M}$ .}
        \label{Fig: convergence}
    \end{figure}
\end{example}

\begin{remark}
    The solution of the problem \eqref{eq: example of optimization 1} described in Example \ref{ex: example of optimization 1} should be regarded only as an illustration of the optimization technique via infinite-dimensional lift: in fact, the kernel $K$ in this example is degenerate and thus the process $X^u$ in \eqref{eq: example of optimization 1} is Markovian. This means that other finite dimensional techniques could have been used in this case. 
\end{remark}

\begin{example}[\emph{fractional and gamma kernels}]\label{ex: example of optimization 2}
    Consider three optimization problems of the form
    \begin{equation}\label{eq: example of optimization 2}
        \begin{cases}
            X_i^u(t) = \int_0^t K_i(t-s) \Big(\alpha u(s)- \beta X^u(s)\Big)ds + \int_0^t K_i(t-s) dW(s),
            \\
            \E\left[X_i^u(T) - \int_0^Tu^2(s)ds\right] \to \max,
        \end{cases}\quad i=1,2,3,
    \end{equation}
    $u\in L^2_a$, where the kernels are chosen as follows: $K_1(t) := t^{0.3}$ (fractional kernel), $K_2(t) := t^{1.1}$ (smooth kernel) and $K_3(t) := e^{-t}t^{0.3}$ (gamma kernel). In these cases, we apply all the machinery presented in section \ref{sec: approximation} to find $\hat u_{n,M}$ for each of the optimal control problems described above. In our simulations, we choose $T=2$, $n=20$, $M=50$; the mesh of the partition for simulating sample paths of $X^u$ is set to be $0.05$, $\sigma = 1$, $X(0) = 0$.
    
    Fig. \ref{Fig: Kernels graphs} depicts approximated optimal controls for different values of $\alpha$ and $\beta$. Note that the gamma kernel $K_3(t)$ (third column) is of particularly interest in optimal advertising. This kernel, in fact, captures the peculiarities of the empirical data (see \cite{Leone1995}) since the past dependence comes into play after a certain amount of time (like a delayed effect) and its relevance declines as time goes forward.

\begin{remark}
     Note that the stochastic Volterra integral equation from \eqref{eq: example of optimization 2} can be sometimes solved explicitly for certain kernels (e.g. via the resolvent method). For instance, the solution $X^u$ which corresponds to the fractional kernel of the type $K(t) = t^h$, $h>0$, and $\beta = 1$ has the form
    \[
        X^u(t) = \Gamma(h+1) \int_0^t (t-s)^h E_{h+1,h+1} \left( -\Gamma(h+1)(t-s)^{h+1} \right) \left( \alpha u(s) ds + dW(s) \right), \quad t\in [0,T],
    \]
    where $E_{a,b}$ again denotes the Mittag-Leffler function. Having the explicit solution, one could solve the optimization problem \eqref{eq: example of optimization 2} by plugging in the shape of $X^u$ to the performance functional and applying the standard minimization techniques in Hilbert spaces. However, as mentioned in the introduction, this leads to some tedious calculations that are complicated to implement, whereas our approach allows to get the approximated solution in a relatively simple manner.
\end{remark}

\begin{figure}[h!]
    \centering
    \begin{minipage}[b]{0.32\linewidth}
        \centering
        (1) $K_1(t)=t^{0.3}$
    \end{minipage}
    \begin{minipage}[b]{0.32\linewidth}
        \centering
        (2) $K_2(t) = t^{1.1}$
    \end{minipage}
    \begin{minipage}[b]{0.32\linewidth}
        \centering
        (3) $K_3(t) = e^{-t} t^{0.3}$
    \end{minipage}
    \begin{minipage}[b]{0.32\linewidth}
        \centering
        \includegraphics[width=\textwidth]{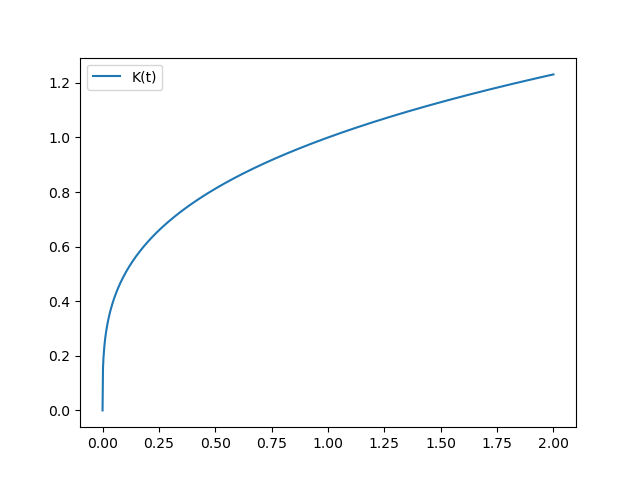}
        (a1)
    \end{minipage}
    \begin{minipage}[b]{0.32\linewidth}
        \centering
        \includegraphics[width=\textwidth]{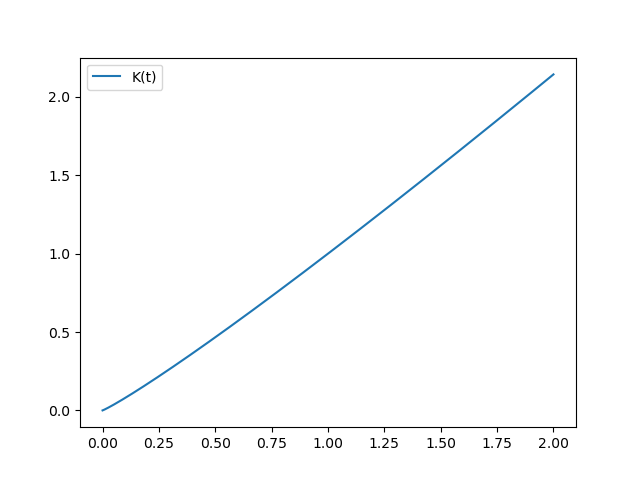}
        (a2)
    \end{minipage}
    \begin{minipage}[b]{0.32\linewidth}
        \centering
        \includegraphics[width=\textwidth]{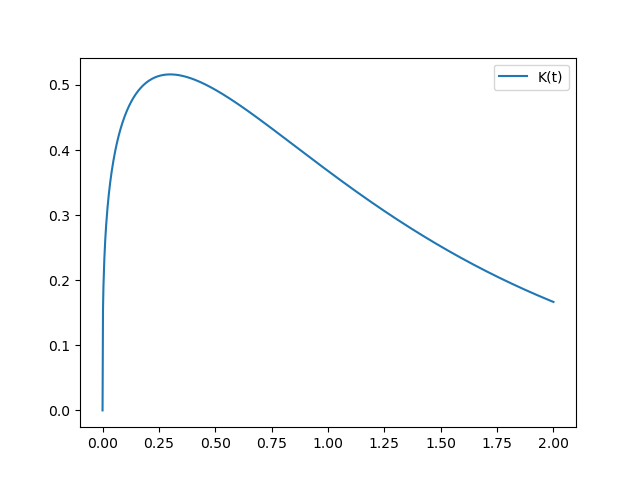}
        (a3)
    \end{minipage}
    \begin{minipage}[b]{0.32\linewidth}
        \centering
        \includegraphics[width=\textwidth]{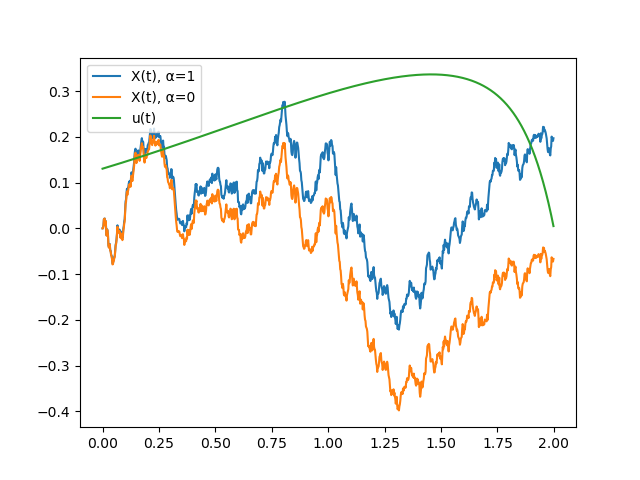}
        (b1)
    \end{minipage}
    \begin{minipage}[b]{0.32\linewidth}
        \centering
        \includegraphics[width=\textwidth]{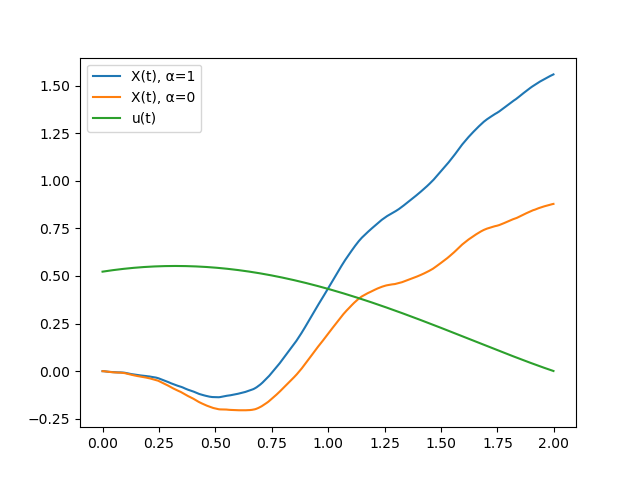}
        (b2)
    \end{minipage}
    \begin{minipage}[b]{0.32\linewidth}
        \centering
        \includegraphics[width=\textwidth]{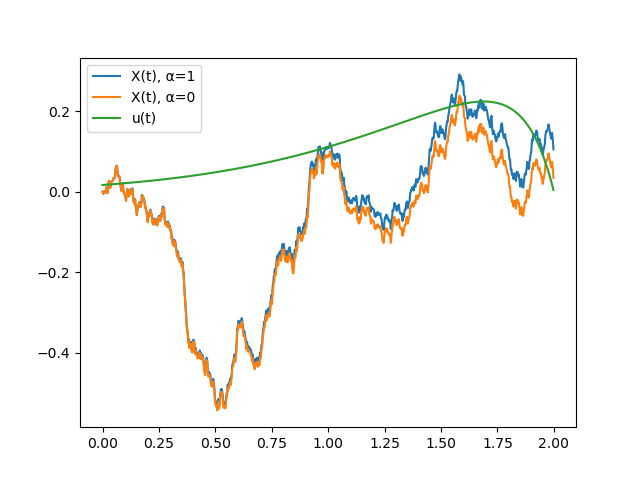}
        (b3)
    \end{minipage}
    \begin{minipage}[b]{0.32\linewidth}
        \centering
        \includegraphics[width=\textwidth]{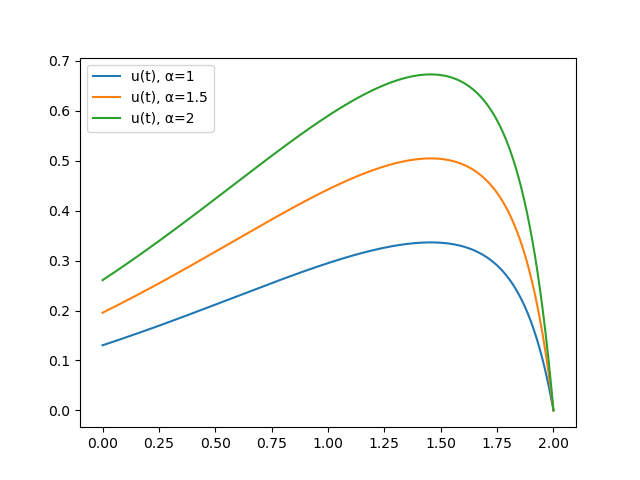}
        (c1)
    \end{minipage}
    \begin{minipage}[b]{0.32\linewidth}
        \centering
        \includegraphics[width=\textwidth]{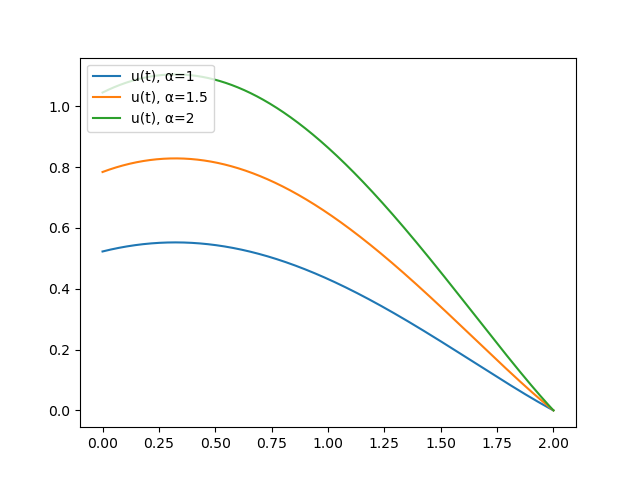}
        (c2)
    \end{minipage}
    \begin{minipage}[b]{0.32\linewidth}
        \centering
        \includegraphics[width=\textwidth]{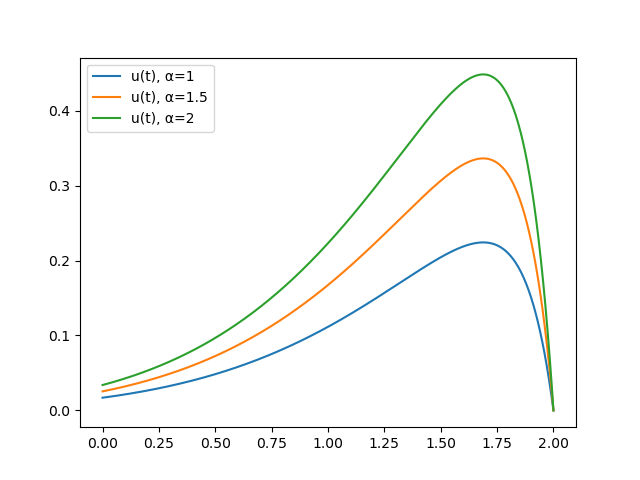}
        (c3)
    \end{minipage}
    \begin{minipage}[b]{0.32\linewidth}
        \centering
        \includegraphics[width=\textwidth]{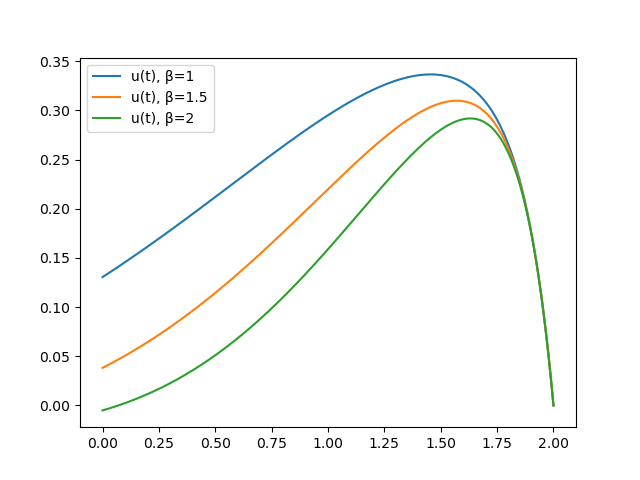}
        (d1)
    \end{minipage}
    \begin{minipage}[b]{0.32\linewidth}
        \centering
        \includegraphics[width=\textwidth]{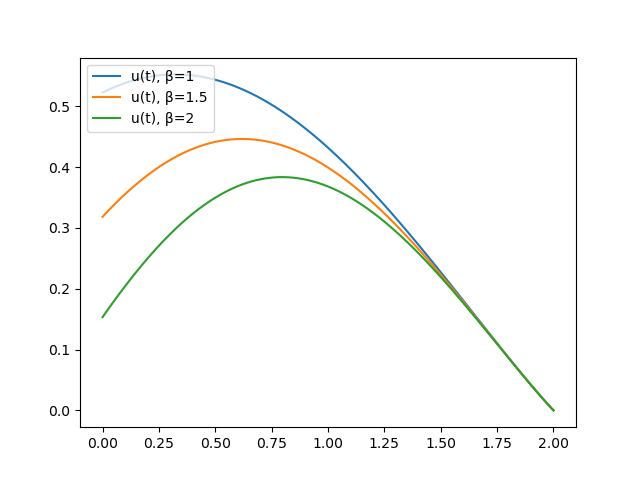}
        (d2)
    \end{minipage}
    \begin{minipage}[b]{0.32\linewidth}
        \centering
        \includegraphics[width=\textwidth]{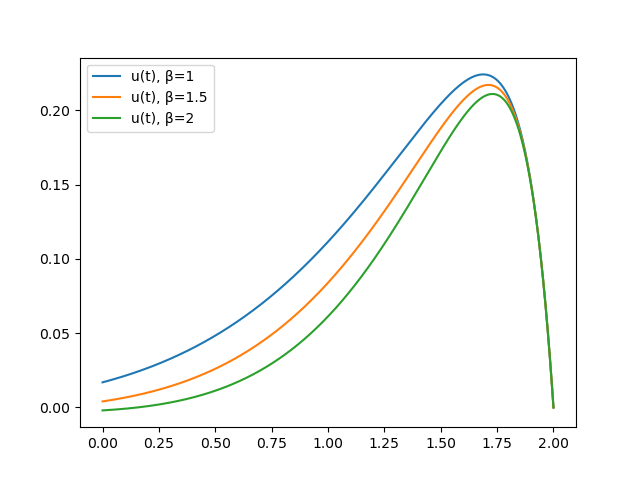}
        (d3)
    \end{minipage}
    \caption{Optimal advertising strategies for control problems with kernels $K_1$--$K_3$ from Example \ref{ex: example of optimization 2}; plots related to the kernel $K_i$ are contained in the $i$th column. Panels (a1)--(a3) depict the graphs of kernels $K_1$--$K_3$; each of (b1)--(b3) represents a sample path of the corresponding $X^u_i(t)$ under optimal control with $\alpha=0$ (orange) and $\alpha=1$ (blue) as well as the approximated optimal control $\hat u_{n,M}$ itself (green). Panels (c1)--(c3) show $\hat u_{n,M}$ for $\alpha=1$ (blue), $\alpha = 1.5$ (orange) and $\alpha = 2$ (green; in all three cases $\beta = 1$), whereas (d1)--(d3) plot the behaviour of $\hat u_{n,M}$ for $\beta=1$ (blue), $\beta = 1.5$ (orange) and $\beta = 2$ (green; in all three cases $\alpha = 1$).}\label{Fig: Kernels graphs}
\end{figure}
\end{example}

\noindent\textbf{Acknowledgments.} Authors would also like to thank Dennis Schroers for the enlightening help with one of the proofs leading to this paper as well as Giulia Di Nunno for the proofreading and valuable remarks.

\newpage
\bibliographystyle{acm}
\bibliography{bib.bib}

\end{document}